\documentclass[12pt]{article}

\usepackage{amsthm}
\usepackage{amsmath}
\usepackage{amssymb}
\usepackage{mathrsfs}

\usepackage{graphicx}

\usepackage[colorlinks=true,citecolor=black,linkcolor=black,urlcolor=blue]{hyperref}

\theoremstyle{plain}
\newtheorem{theorem}{Theorem}
\newtheorem{lemma}[theorem]{Lemma}
\newtheorem{corollary}[theorem]{Corollary}
\newtheorem{proposition}[theorem]{Proposition}

\theoremstyle{definition}

\theoremstyle{remark}

\newcommand{\beq}[1]{\begin{equation}\label{#1}}
\newcommand{\enq}[0]{\end{equation}}

\newcommand{\C}[2]{{{#1}\choose{{#2}}}}
\newcommand{\ga}[0]{\alpha }

\newcommand{\gO}[0]{\Omega}

\newcommand{\eps}[0]{\varepsilon }

\newcommand{\ra}[0]{\rightarrow}

\newcommand{\CC}[0]{{\bf C}}
\newcommand{\ZZ}[0]{{\bf Z}}

\newcommand{\GG}[0]{{\bf G}}

\newcommand{\zz}[0]{\mathbf{z}}

\newcommand{\D}[0]{{\cal D}}
\newcommand{\eee}[0]{{\cal E}}
\newcommand{\f}[0]{{\cal F}}
\newcommand{\g}[0]{{\cal G}}

\newcommand{\sub}[0]{\subseteq}
\newcommand{\sm}[0]{\setminus}
\renewcommand{\dots}[0]{,\ldots,}

\newcommand{\ov}[0]{\overline}

\newcommand{\uzero}[0]{\underline{0}}

\newcommand{\E}[0]{{\sf E}}

\title{\bf Modular statistics for subgraph counts \\in
sparse random graphs}

\author{Bobby DeMarco\thanks{Supported by NSF grant DMS1201337.}\\
\small\tt rvdemarco@gmail.com
\and
Jeff Kahn\thanks{Supported by NSF grant DMS1201337.} \\
\small Department of Mathematics\\[-0.8ex]
\small Rutgers University\\[-0.8ex]
\small Piscataway, NJ, USA\\
\small\tt jkahn@math.rutgers.edu
\and
Amanda Redlich\thanks{Supported by NSF Award No. 1004382.}\\
\small Department of Mathematics\\[-0.8ex]
\small Bowdoin College\\[-0.8ex]
\small Brunswick, ME, USA\\
\small\tt aredlich@bowdoin.edu
}

\begin{document}

\maketitle

\begin{abstract}
Answering a question of Kolaitis and Kopparty, we show
that, for given integer $q>1$ and
pairwise nonisomorphic connected graphs
$G_1\dots G_k$,
if $p=p(n) $ is such that
$\Pr(G_{n,p}\supseteq G_i)\rightarrow 1$ $\forall i$,
then, with $\xi_i$ the number of copies of $G_i$ in
$G_{n,p}$, $(\xi_1\dots \xi_k)$ is asymptotically uniformly
distributed on ${\bf Z}_q^k$.

\end{abstract}

\section{Introduction}
For graphs $G, H$ write $N(G,H)$ for the number of
unlabeled copies of $H$ in $G$ (e.g. $N(K_r,K_s) = \C{r}{s}$).
We use both $G_{n,p}$ and $G(n,p)$
for the ordinary (``binomial"
or ``Erd\H{o}s-R\'enyi") random graph.

We are interested here in extending to nonconstant $p$
the following beautiful result of
Kolaitis and Kopparty \cite{kk}.

\begin{theorem}\label{TKK}
Fix an integer $q>1$,
$p\in (0,1)$ and pairwise nonisomorphic connected
graphs $G_1\dots G_k$, each with at least two vertices,
and let $\xi_i $ be $ N(G_{n,p},G_i) \pmod{q}$.
Then the distribution of $\xi = (\xi_1\dots \xi_k)$ is
$e^{-\gO(n)}$-close to uniform on $\ZZ_q^k$.
In particular, for each $a\in \ZZ_q^k$,
$
\Pr(\xi =a)\ra q^{-k} $ as $n\rightarrow \infty$.
\end{theorem}
\noindent
(Recall two distributions are $\eps$-{\em close}
if their statistical (a.k.a. variation) distance is at most $\eps$.)  Essentially, this theorem states that for constants $p$ and $q$, subgraphs of $G(n,p)$ are uniformly distributed modulo $q$.

Theorem \ref{TKK} was motivated by an application to 0-1 laws for
first order logic with a parity quantifier or, more
generally, a quantifier that allows counting modulo $q$;
see Section \ref{Discussion} for a little more on this.

A natural question raised in \cite{kk}
(and communicated to the authors by S.K.)
asks, to what extent does Theorem \ref{TKK}
remain true if $p$ is allowed to tend to zero as $n $ grows,
e.g. if $p=n^{-\ga}$ for some fixed $\ga>0$?
Our purpose here is to answer this question.

We need a little notation.
For a graph $H=(V,E)$, set $v_H=|V|$, $e_H=|H|:=|E|$,
$\rho(H)= e_H/v_H$ and $m(H) = \max\{\rho(H'):H'\sub H, v_{H'}>0\}$.
Recall (see e.g. \cite{JLR}) that $n^{-1/m(H)}$ is a
threshold function for containment of $H$; that is,
the probability that $G_{n,p}$ ($p=p(n)$)
contains a copy of $H$
tends to 0 if $pn^{1/m(H)}\ra 0$ and
to 1 if $pn^{1/m(H)}\ra \infty$.
Given a collection $\g$ of graphs, set
$m(\g) =\max\{m(G):G\in\g\}$,
$p_\g(n) = n^{-1/m(\g)}$ and
$$
\Phi_\g(n,p) = \min_{G\in \g}\min\{n^{v_H}p^{e_H}:H\sub G ,v_H>0\}.
$$

\begin{theorem}\label{MT}
Let q, $G_1\dots G_k$ and $\xi=(\xi_1\dots \xi_k)$ be
as in Theorem \ref{TKK} and $\g=\{G_1\dots G_k\}$.
If $p=\omega(p_\g(n))$, then
the distribution of $\xi$ is
$\exp[-\gO(\Phi_\g(n,p))]$-close to uniform on $\ZZ_q^k$.
\end{theorem}
\noindent
(Of course the constant in
the exponent depends on $q$ and $\g$.)

Suppose e.g. that $q=k=2$, $G_1=K_3$, and $G_2=K_4$.  
Then $m(\g)=m(G_2)=3/2$ ($m(G_1)=1$) and $p_\g(n)=n^{-2/3}$, so the 
theorem says that, asymptotically speaking, the parities of the 
numbers of copies of $K_3$ and $K_4$ are 
independent with each equally likely to be even or odd, provided $p=\omega(n^{-2/3})$.

For the special case $\g=\{K_3\}$, a somewhat weaker
version of Theorem~\ref{MT}---with $\exp[-\gO(\Phi_\g(n,p))]$
replaced by something polynomial in $n$ and $p$---has been shown
by Noga Alon \cite{SKPC}.

\medskip
We should also note here an immediate consequence of
Theorem \ref{MT}, which again answers a question from \cite{kk}.

\begin{corollary}\label{Cor}
Let $q$, $\g$ be as in Theorem \ref{TKK}, fix a positive
irrational $\ga$, and let $I=\{i\in [k]: m(G_i) < \ga^{-1}\}$
and $J=[k]\sm I$.  Then for $p=n^{-\ga}$ and $a\in \ZZ_q^k$
(and $\xi$ as in Theorem \ref{TKK}),
$$\Pr(\xi=a)\ra \left\{\begin{array}{ll}
q^{-|I|}&\mbox{if $a_j=0 ~\forall j\in J$,}\\
0&\mbox{otherwise.}
\end{array}\right.
$$
\end{corollary}

\noindent
This is of interest partly for its possible relevance
to proving a modular convergence law (again see Section \ref{Discussion})
for $p=n^{-\ga}$ with $\ga$ irrational
({\em cf.} \cite[Theorem 6]{ss}, which says that
for such $p$ a 0-1 law holds for any first order property);
but we also have, again from \cite{kk}:
``Even the behavior of subgraph frequencies mod 2 in this
setting [i.e. with $p$ as in Corollary \ref{Cor}]
seems quite intriguing."

\medskip
The proof of Theorem \ref{MT}, given in
the next section, is similar to that of
Theorem \ref{TKK} in \cite{kk}.  In truth, we just
add one little idea to the machinery of \cite{kk};
nonetheless, as the proof answers a rather basic question,
and was apparently not quite trivial to find,
it seems worth recording.

\section{Proof}\label{Proof}

We will need the following two facts, the first of which,
from \cite{kk}, generalizes a result of
Babai, Nisan and Szegedy \cite{bns}.

\begin{lemma}\label{poly}
Let $q>1$ and $d>0$ be integers and $p \in (0,1)$.
Let $\mathcal{F}\subseteq 2^{[m]}$ and
let $Q(z_{1}, \ldots z_{m})\in\ZZ_{q}[z_{1},\ldots z_{m}]$
be a polynomial of the form
$$\sum_{S \in \mathcal{F}}a_{S}
\prod_{i \in S}z_{i}+Q'(z_{1},\ldots z_{m}),$$
where $\deg (Q')<d$.
Suppose there is some
$\mathcal{E}=\{E_{1},\ldots E_{r}\}\subseteq \mathcal{F}$
such that

\begin{itemize}
\item $|E_{j}|=d$ for all $j$,
\item $a_{E_{j}}\neq 0$ for all $j$,
\item $E_{j}\cap E_{j'}=\emptyset$ for all $j\neq j'$, and
\item for each $S \in \mathcal{F}\backslash \mathcal{E}$, $|S\cap(\cup_{j}E_{j})|<d$.
\end{itemize}
Let $\mathbf{z}=(\mathbf{z}_{1},\ldots \mathbf{z}_{m})
\in\ZZ^{m}_{q}$
be the random variable where, independently for each $i$,
$\Pr(\mathbf{z}_{i}=1)=p$ and $\Pr(\mathbf{z}_{i}=0)=1-p$.
Then
for $\omega \in \CC$ a primitive $q^{th}$-root of unity,
\beq{Eomega}
|\E[\omega^{Q(\mathbf{z})}]|\leq 2^{-\Omega(r)}.
\enq
\end{lemma}
\noindent
(We again observe
that the implied constant in the $\Omega(r)$ term depends on $q,p$ and $d$.)

\begin{lemma}[``Vazirani XOR Lemma"]\label{xor}
Let $q>1$ be an integer and $\omega \in \CC$
a primitive $q^{th}$-root of unity.
Let $\xi=(\xi_1, \ldots, \xi_l)$ be a random variable
taking values in
$\ZZ_{q}^{l}$.
Suppose that for every nonzero $c \in \ZZ^{l}_{q}$,
$$|\E[\omega^{\sum c_i \xi_i}]|\leq \epsilon.$$
Then the distribution of $\xi$ is $(q^{l}\epsilon)$-close to
uniform on $\ZZ^{l}_{q}$.
\end{lemma}

\begin{proof}[Proof of Theorem \ref{MT}]
Letting $e$ run over edges of $K_n$, the argument of
\cite{kk} expresses each $\sum c_i\xi_i$ in the natural way
as a polynomial in the indicators
$\mathbf{z}_e:={\bf 1}_{\{e\in G(n,p)\}}$ ($e\in E(K_n)$)---namely,
$$
\sum_i c_i\xi_i =\sum_i c_i\sum\{\prod_{e\in H}\mathbf{z}_e:
G_i\cong H\sub K_n\}
$$
---and for the $\eee$ of Lemma \ref{poly} uses $\gO(n)$
vertex-disjoint
copies of some largest $G_i$ among those with
$c_i\neq 0$.
The problem with this in the present situation is the
(hidden) dependence of the bound in \eqref{Eomega} on $p$.

We get around this difficulty by choosing our random
graph in two steps, so that when we come to apply
Lemma \ref{poly} we are back to constant $p$.
For simplicity we now write $\Phi$ for $\Phi_\g(n,p)$,
$\GG '$ for $G(n,2p)$ and
$\GG $ for the random subgraph of ${\bf G'}$ in which each edge is
present, independently of other choices, with probability 1/2;
in particular, our $\xi_i$'s are functions of
$\GG $ ($=G(n,p)$).

\medskip
Given $\GG'$,
we will apply Lemma \ref{poly} with variables
$\mathbf{z}_e ={\bf 1}_{\{e\in \GG \}}$ ($e\in \GG'$),
$\f$ the collection of copies of $G_1\dots G_k$
in $\GG'$, and $\mathcal{E}\sub \f$ a large collection of
vertex-disjoint copies of an appropriate $G_i$;
so first of all we need existence of such an $\eee$.
For a given $\eps$, let
$\D=\D_\eps $ be the event that $\GG'$ contains, for each $i$,
a collection of
$r:=\eps \Phi$ vertex-disjoint copies of $G_i$.
\begin{proposition}\label{Dprop}
There is a fixed $\eps>0$ (depending on $\g$) for which
\beq{PrD}
\Pr(\ov{\D}) < \exp[-\gO(\Phi)].
\enq
\end{proposition}

\noindent
{\em Proof.}

Though we don't know a reference,
this is presumably not new
and the ideas needed to prove it may all be found in
\cite{JLR};

so we just indicate what's involved.

Fix $i\in [k]$ and write $H$ for $G_i$.
Let $Y$ be the maximum size of a
collection of disjoint copies of $H$ in $\GG'$.
It is enough to show that
the (more properly, ``a") {\em median} of $Y$ is $\gO( \Phi)$;
\eqref{PrD} then follows {\em via} an inequality of
Talagrand (\cite{Talagrand} or \cite[Theorem 2.29]{JLR})
as in the argument for the
edge-disjoint analogue of Proposition~\ref{Dprop}
given on page 77 of \cite{JLR}.
(In our case Talagrand's inequality says that for a median $m$ of $Y$
and $t>0$, $\Pr(Y\leq m-t) \leq 2\exp[-t^2/(4\psi(m))]$,
where $\psi(r) =r|H|$.)

For a lower bound on the median of $Y$, write $X$ for
the number of copies of $H$ (in $\GG'$) and $Z$ for the number
of (unordered) pairs of non-disjoint copies.  Then:

\medskip\noindent
(i) $\E(X) =\gO(\Phi)$ (this is immediate from the definitions);

\medskip\noindent
(ii) w.h.p. $X>(1-o(1))\E X$ (a basic application of the 2nd
moment method; see \cite[Remark 3.7]{JLR});

\medskip\noindent
(iii)  $\E Z <c\E^2X/\Phi$ for a suitable
fixed $c$ (a straightforward calculation using the definition
of $\Phi$), so with probability at least 3/4,
$Z<4c\E^2X/\Phi$;

\medskip\noindent
(iv) by Tur\'an's Theorem (applied to the graph with vertices the 
copies of $H$, edges the non-disjoint pairs and (therefore)
independence number $Y$;
{\em cf.} \cite[Eq. (3.21)]{JLR}),
$Y\geq X^2/(X+2Z)$; and thus 

\medskip\noindent
(v) with probability
at least $3/4-o(1)$,
\[
Y > \frac{(1-o(1))\E^2X}{\E X + 8c\E^2X/\Phi}
=\gO(\Phi)
\]
(where the first inequality uses the fact that $x^2/(x+2z)$ is
increasing in $x$ for $x,z>0$).\qed

\medskip
In view of Proposition~\ref{Dprop} it is enough to show that
for any $G'$ satisfying $\D$,
the conditional distribution of $\xi$ given
$\{{\bf G'}=G'\}$ is
$\exp[-\Omega(\Phi)]$-close to uniform on $\ZZ_q^k$.
Given such a $G'$
and $\uzero\neq c\in \ZZ_q^k$,
take $\f_i$ to consist of all copies
of $G_i$ in $G'$ ($i\in [k]$) and $\f=\cup \{\f_i:c_i\neq 0\}$.
Fix, in addition, some
$i_0\in [k]$ with $c_{i_0}\neq 0$ and
$|G_{i_0}|=\max\{|G_i|:c_i\neq 0\}=:d$, and
some
$\eee=\{E_1\dots E_r\}\sub \f_{i_0}$, with the
$E_i$'s vertex-disjoint.

We have
$$
\sum_{i\in [k]} c_i\xi_i
= \sum_{i\in [k]}c_i\sum_{H\in \f_i}\prod_{e\in H}\zz_e
=: Q(\zz),
$$
where $\zz_e = {\bf 1}_{\{e\in \GG\}}$ for $e\in G'$.
We then need to say that $Q$, $\f$ and $\eee$
(with $q, d$ and $p=1/2$) satisfy the requirements of Lemma \ref{poly}.
But the first three of these are immediate
and the fourth follows from the connectivity of the $G_i$'s:
for $H\in \f\sm\eee$, if $V(H)\not\sub V(E_i) ~\forall i$, then
(since $H$ is connected and the $E_i$'s are
vertex-disjoint) $H\not\sub \cup E_i$,
whence $|H\cap (\cup E_i)|<|H|\leq d$;
otherwise we have $V(H)\sub V(E_j)$ for some $j$ and,
since $H\neq E_j$,
$|H\cap (\cup E_i)| = |H\cap E_j|< |E_i|=d$.
Thus Lemma \ref{poly} applies, yielding
\beq{Eomega'}
|\E ~\omega^{Q(z)}| \leq \exp[-\Omega(\Phi)],
\enq
and then (since this was for any $c\neq\uzero$) Lemma \ref{xor}
says that, as desired,
the conditional distribution of $\xi$ given
$\{{\bf G'}=G'\}$ is
$\exp[-\Omega(\Phi)]$-close to uniform on $\ZZ_q^k$.

\bigskip\noindent
$~$

\end{proof}

\section{Discussion}\label{Discussion}

As mentioned earlier, Theorem \ref{TKK} is a key
ingredient in the proof
of the
Kolaitis-Kopparty ``modular
convergence law" for
first order logic with a parity quantifier, or, more
generally, a quantifier that allows counting mod $q$.
This law says, briefly, that, for fixed $p$ and $n\ra\infty$,
the probability
of a given sentence in the system under consideration
tends to a limit that depends only on the congruence
class of $n$ mod $q$.
(See also \cite{strange} for an in-depth discussion
of 0-1 laws for random graphs.)

As suggested in \cite{kk}, it would be interesting to understand
to what extent such a law holds in the sparse setting.
Theorem \ref{MT} gets about half way to this goal
(for $p$ in its range);
but the other half---an assertion like
Theorem 2.3 of \cite{kk} to the effect that all relevant information
is contained in the subgraph frequencies---seems
to require something
new, since
the quantifier elimination process underlying that step
depends critically on properties of $G(n,p)$ that hold
for constant $p$ but fail when $p$ tends to zero.

\medskip
In closing we just mention that it would be
interesting to find a proof of
Theorem \ref{MT} that proceeds from first principles and does
not depend on the ``generalized inner product" polynomials
underlying Lemma \ref{poly}.

\subsection*{Acknowledgement}
We would like to thank Swastik Kopparty for telling us
the problem and for helpful conversations on the material
of Section \ref{Discussion}.


\begin{thebibliography}{10}

\bibitem{bns}
L. Babai, N. Nisan and M. Szegedy,
Multiparty protocols and logspace-hard pseudorandom sequences,
in {\em Proc. 21st ACM Symposium on the Theory of Computing}, 1989.

\bibitem{JLR}
S. Janson, T. \L uczak and A. Ruci\'nski,
\textit{Random Graphs}, Wiley-Interscience, 2000.

\bibitem{SKPC}
S. Kopparty, personal communication.

\bibitem{kk}
P. G. Kolaitis and S. Kopparty, \textit{Random graphs and the parity quantifier}, J. ACM, vol. 60 (5) Article 37, 2013.

\bibitem{ss}
S. Shelah and J. Spencer,
Zero-one laws for sparse random graphs,
{\em J. Amer. Math. Soc.} {\bf 1} (1988), 97-115.



\bibitem{strange}
J. Spencer, \textit{The Strange Logic of Random Graphs}, Springer, 2001


\bibitem{Talagrand} M. Talagrand,
Concentration of measure and isoperimetric inequalities in product
spaces, {\em Inst. Hautes \'Etudes Sci. Publ. Math.}
{\bf 81} (1995), 73-205.

\end{thebibliography}
\end{document}